\documentclass{article}%
\usepackage{amsfonts}
\usepackage{amsmath}
\usepackage{amssymb}
\usepackage{graphicx}%
\providecommand{\U}[1]{\protect\rule{.1in}{.1in}}
\newtheorem{theorem}{Theorem}

\newtheorem{corollary}[theorem]{Corollary}

\newtheorem{lemma}[theorem]{Lemma}

\newenvironment{proof}[1][Proof]{\noindent\textbf{#1.} }{\ \rule{0.5em}{0.5em}}
\begin{document}

\title{ }

\begin{center}
{\Large Some formulae for products of Fubini polynomials with applications}%
\bigskip

Levent Karg\i n

Akseki Vocational School, Alanya Alaaddin Keykubat University,\newline Antalya
TR-07630, Turkey\newline

\textbf{E-mail: }leventkargin48@gmail.com\bigskip

\textbf{Abstract}
\end{center}

In this paper we evaluate sums and integrals of products of Fubini polynomials
and have new explicit formulas for Fubini polynomials and numbers. As a
consequence of these results new explicit formulas for $p$-Bernoulli numbers
and Apostol-Bernoulli functions are given. Besides, integrals of products of
Apostol-Bernoulli functions are derived.

\textbf{2000 Mathematics Subject Classification: }11B68, 11B75, 11B83.

\textbf{Key words: }Fubini numbers and polynomials, Apostol-Bernoulli
functions, $p$-Bernoulli numbers.

\section{Introduction, Definitions and Notations}

Let $%
\genfrac{\{}{\}}{0pt}{}{n}{k}%
$ be the Stirling numbers of the second kind (\cite{Graham}). Fubini
polynomials are defined by \cite{T}
\begin{equation}
F_{n}\left(  y\right)  =\sum_{k=0}^{n}%
\genfrac{\{}{\}}{0pt}{}{n}{k}%
k!y^{k}. \label{2}%
\end{equation}
They have the exponential generating function
\begin{equation}
\frac{1}{1-y\left(  e^{t}-1\right)  }=\sum_{n=0}^{\infty}F_{n}\left(
y\right)  \frac{t^{n}}{n!}\text{,} \label{10}%
\end{equation}
and are related to the geometric series \cite{B}%
\[
\left(  y\frac{d}{dy}\right)  ^{m}\frac{1}{1-y}=\sum_{k=0}^{\infty}k^{m}%
y^{k}=\frac{1}{1-y}F_{m}\left(  \frac{y}{1-y}\right)  ,\text{ \ }\left\vert
y\right\vert <1.
\]
Because of this relation Fubini polynomials are also called geometric
polynomials. In addition, the following recurrence relation holds for the
Fubini polynomials \cite{Diletal}%
\begin{equation}
F_{n+1}\left(  y\right)  =y\frac{d}{dy}\left[  F_{n}\left(  y\right)
+yF_{n}\left(  y\right)  \right]  . \label{11}%
\end{equation}

The $n$ th Fubini number (ordered Bell number or geometric number) \cite{Da,
Gr, T}, $F_{n},$ is defined by
\begin{equation}
F_{n}\left(  1\right)  =F_{n}=\sum_{k=0}^{n}%
\genfrac{\{}{\}}{0pt}{}{n}{k}%
k!\text{,} \label{14}%
\end{equation}
and counts all the possible set partitions of an $n$ element set such that the
order of the blocks matters. Besides with this combinatorial property, these
numbers are seen in the evaluation of the following series
\begin{equation}
\sum_{k=0}^{\infty}\frac{k^{n}}{2^{k}}=2F_{n}. \label{1}%
\end{equation}

In the literature numerous identities concerned with these polynomials and
numbers were obtained \cite{B, B2, B3, B4, Dil1, MEZO} and their
generalizations are given \cite{B, Dil2, Kargin, Mezo2}. The main purpose of
this paper is to generalize the binomial formulas \cite{Diletal}%
\begin{align}
\sum_{k=0}^{n}\binom{n}{k}F_{k}  &  =2F_{n},\text{ \ }n>0,\label{5}\\
2\sum_{k=0}^{n}\binom{n}{k}\left(  -1\right)  ^{k}F_{k}  &  =\left(
-1\right)  ^{n}F_{n}+1,\text{ \ }n\geq0, \label{6}%
\end{align}
and the integral representation \cite{KELLER}%
\begin{equation}
\int_{-1}^{0}F_{n}\left(  y\right)  dy=B_{n},\text{ \ }n>0. \label{26}%
\end{equation}
Here $B_{n}$ is the Bernoulli numbers defined by the explicit formula
\begin{equation}
B_{n}=\sum_{k=0}^{n}%
\genfrac{\{}{\}}{0pt}{}{n}{k}%
\left(  -1\right)  ^{k}\frac{k!}{k+1}. \label{32}%
\end{equation}
As applications of these generalizations we obtain explicit formulas for
Apostol-Bernoulli functions and $p$-Bernoulli numbers and integrals of
products of Apostol-Bernoulli functions. We use generating function technique
in the proofs.

Now we state our results.

\section{Sums of products of Fubini polynomials}

In this section we define two variable Fubini polynomials and obtain some
basic properties which give us new formulas for $F_{n}\left(  y\right)  .$
Moreover we shall consider the sums of products of two Fubini polynomials. The
sums of products of various polynomials and numbers with or without binomial
coefficients have been studied (e.g., \cite{Agoh, Kamano, Komatsu, Komatsu2,
Singh, Zhao}).

Two variable Fubini polynomials are defined by means of the following
generating function%
\begin{equation}
\sum_{n=0}^{\infty}F_{n}\left(  x;y\right)  \frac{t^{n}}{n!}=\frac{e^{xt}%
}{1-y\left(  e^{t}-1\right)  }. \label{8}%
\end{equation}
For some special cases of (\ref{8}), we have%
\begin{equation}
F_{n}\left(  0;y\right)  =F_{n}\left(  y\right)  \text{ and }F_{n}\left(
0;1\right)  =F_{n}. \label{29}%
\end{equation}

We can rewrite (\ref{8}) as
\begin{align*}
\sum_{n=0}^{\infty}F_{n}\left(  x;y\right)  \frac{t^{n}}{n!}  &  =\frac
{1}{1-y\left(  e^{t}-1\right)  }e^{xt}\\
&  =\sum_{n=0}^{\infty}F_{n}\left(  y\right)  \frac{t^{n}}{n!}\sum
_{n=0}^{\infty}x^{n}\frac{t^{n}}{n!}\\
&  =\sum_{n=0}^{\infty}\left[  \sum_{k=0}^{n}\binom{n}{k}F_{k}\left(
y\right)  x^{n-k}\right]  \frac{t^{n}}{n!}.
\end{align*}
Comparing the coefficients of $\frac{t^{n}}{n!}$ yields%
\begin{equation}
F_{n}\left(  x;y\right)  =\sum_{k=0}^{n}\binom{n}{k}F_{k}\left(  y\right)
x^{n-k}. \label{3}%
\end{equation}

From (\ref{8}) we have
\begin{align*}
\sum_{n=0}^{\infty}\left[  F_{n}\left(  x+1;y\right)  -F_{n}\left(
x;y\right)  \right]  \frac{t^{n}}{n!}  &  =\frac{e^{xt}\left(  e^{t}-1\right)
}{1-y\left(  e^{t}-1\right)  }\\
&  =\frac{1}{y}\left[  \frac{e^{xt}}{1-y\left(  e^{t}-1\right)  }%
-e^{xt}\right] \\
&  =\frac{1}{y}\sum_{n=0}^{\infty}\left[  F_{n}\left(  x;y\right)
-x^{n}\right]  \frac{t^{n}}{n!}.
\end{align*}
Comparing the coefficients of $\frac{t^{n}}{n!}$ gives
\begin{equation}
yF_{n}\left(  x+1;y\right)  =\left(  1+y\right)  F_{n}\left(  x;y\right)
-x^{n}. \label{4}%
\end{equation}
Thus, setting $x=0$ and $x=-1$ in (\ref{4}) we find
\begin{align}
yF_{n}\left(  1;y\right)   &  =\left(  1+y\right)  F_{n}\left(  y\right)
,\text{ \ }n>0,\label{7}\\
\left(  1+y\right)  F_{n}\left(  -1;y\right)   &  =yF_{n}\left(  y\right)
+\left(  -1\right)  ^{n},\text{ \ }n\geq0, \label{9}%
\end{align}
respectively. Combining these relations with (\ref{3}) gives the equations
(\ref{5}) and (\ref{6}) which were obtained by using Euler-Siedel matrix
method in \cite{Diletal}.

Now, we want to give the generalization of the binomial formula (\ref{5}).
Derivative of (\ref{8}) can be written as%
\[
\frac{\partial}{\partial t}\left(  \frac{e^{xt}}{1-y\left(  e^{t}-1\right)
}\right)  =\frac{xe^{xt}}{1-y\left(  e^{t}-1\right)  }+\frac{ye^{t}%
}{1-y\left(  e^{t}-1\right)  }\frac{e^{xt}}{1-y\left(  e^{t}-1\right)  }.
\]
Taking $x=x_{1}+x_{2}-1$ leads%
\begin{align*}
\frac{\partial}{\partial t}\left(  \frac{e^{xt}}{1-y\left(  e^{t}-1\right)
}\right)   &  =\sum_{n=0}^{\infty}F_{n+1}\left(  x_{1}+x_{2}-1;y\right)
\frac{t^{n}}{n!},\\
\frac{xe^{xt}}{1-y\left(  e^{t}-1\right)  }  &  =\left(  x_{1}+x_{2}-1\right)
\sum_{n=0}^{\infty}F_{n}\left(  x_{1}+x_{2}-1;y\right)  \frac{t^{n}}{n!}%
\end{align*}
and
\begin{align*}
\frac{ye^{t}}{1-y\left(  e^{t}-1\right)  }\frac{e^{xt}}{1-y\left(
e^{t}-1\right)  }  &  =y\left(  \sum_{n=0}^{\infty}F_{n}\left(  x_{1}%
;y\right)  \frac{t^{n}}{n!}\right)  \left(  \sum_{n=0}^{\infty}F_{n}\left(
x_{2};y\right)  \frac{t^{n}}{n!}\right) \\
&  =y\sum_{n=0}^{\infty}\sum_{k=0}^{n}\binom{n}{k}F_{k}\left(  x_{1};y\right)
F_{n-k}\left(  x_{2};y\right)  \frac{t^{n}}{n!}.
\end{align*}
By equating the coefficients of $\frac{t^{n}}{n!},$ we get
\[
y\sum_{k=0}^{n}\binom{n}{k}F_{k}\left(  x_{1};y\right)  F_{n-k}\left(
x_{2};y\right)  =F_{n+1}\left(  x_{1}+x_{2}-1;y\right)  -\left(  x_{1}%
+x_{2}-1\right)  F_{n}\left(  x_{1}+x_{2}-1;y\right)  .
\]
For $x_{1}=x_{2}=0$ in the above equation, using (\ref{9}) give the sums of
products of the Fubini polynomials.

\begin{theorem}
\label{teo1}For $n\geq0,$%
\begin{equation}
\left(  y+1\right)  \sum_{k=0}^{n}\binom{n}{k}F_{k}\left(  y\right)
F_{n-k}\left(  y\right)  =F_{n+1}\left(  y\right)  +F_{n}\left(  y\right)  .
\label{13}%
\end{equation}
When $y=1$ this becomes
\begin{equation}
2\sum_{k=0}^{n}\binom{n}{k}F_{k}F_{n-k}=F_{n+1}+F_{n}. \label{12}%
\end{equation}

\end{theorem}

Now, we investigate the sums of products of the Fubini polynomials for
different $y$ values in the following theorem.

\begin{theorem}
For $n\geq0$ and $y_{1}\neq y_{2}$,%
\begin{equation}
\sum_{k=0}^{n}\binom{n}{k}F_{k}\left(  y_{1}\right)  F_{n-k}\left(
y_{2}\right)  =\frac{y_{2}F_{n}\left(  y_{2}\right)  -y_{1}F_{n}\left(
y_{1}\right)  }{y_{2}-y_{1}}. \label{23}%
\end{equation}
\bigskip
\end{theorem}

\begin{proof}
The products of (\ref{8}) can be written as%
\begin{align}
&  \frac{e^{x_{1}t}}{\left(  1-y_{1}\left(  e^{t}-1\right)  \right)  }%
\frac{e^{x_{2}t}}{\left(  1-y_{2}\left(  e^{t}-1\right)  \right)  }%
\label{69}\\
&  \qquad\qquad=\frac{y_{2}}{y_{2}-y_{1}}\frac{e^{\left(  x_{1}+x_{2}\right)
t}}{1-y_{2}\left(  e^{t}-1\right)  }-\frac{y_{1}}{y_{2}-y_{1}}\frac{e^{\left(
x_{1}+x_{2}\right)  t}}{1-y_{1}\left(  e^{t}-1\right)  }.\nonumber
\end{align}
Using the same method as in the proof of Theorem \ref{teo1} we have{}%
\[
\sum_{k=0}^{n}\binom{n}{k}F_{k}\left(  x_{1};y_{1}\right)  F_{n-k}\left(
x_{2};y_{2}\right)  =\frac{y_{2}F_{n}\left(  x_{1}+x_{2};y_{2}\right)
-y_{1}F_{n}\left(  x_{1}+x_{2};y_{1}\right)  }{y_{2}-y_{1}}.
\]
For $x_{1}=x_{2}=0$ in the above equation gives the desired equation.
\end{proof}

As we know, for $y=1$ Fubini polynomials reduce to Fubini numbers. We now
point out (see \eqref{15}) that Fubini numbers arise for other value of $y$,
too. If we take $y-1$ in place of $y$ in (\ref{8}) we have
\begin{equation}
F_{n}\left(  x;y-1\right)  =\left(  -1\right)  ^{n}F_{n}\left(  1-x;-y\right)
. \label{18}%
\end{equation}
Setting $x=0$ in the above equation and using the relation (\ref{7}) we have
the reflection formula%

\begin{equation}
F_{n}\left(  y\right)  =\left(  -1\right)  ^{n}\frac{y}{y+1}F_{n}\left(
-y-1\right)  ,\text{ \ }n>0. \label{19}%
\end{equation}
Therefore, using (\ref{2}) gives a new explicit formula for Fubini polynomials
in the following theorem.

\begin{theorem}
\label{teo4}For $n>0$ we obtain
\begin{equation}
F_{n}\left(  y\right)  =y\sum_{k=1}^{n}%
\genfrac{\{}{\}}{0pt}{}{n}{k}%
\left(  -1\right)  ^{n+k}k!\left(  y+1\right)  ^{k-1}. \label{21}%
\end{equation}

\end{theorem}

Note that, when $y=1,$ (\ref{21}) reduce to \cite[Thereom 4.2]{GuoandQi}.
Moreover, from (\ref{19}) we get two conclusion as%
\begin{equation}
F_{2k}\left(  \frac{-1}{2}\right)  =0\text{ and }F_{n}\left(  -2\right)
=\left(  -1\right)  ^{n}2F_{n}. \label{15}%
\end{equation}
Thus, if we take $y_{1}=-2$ and $y_{2}=1$ in (\ref{23}) and use the second
part of (\ref{15}), we obtain the alternating sums of products of Fubini numbers.

\begin{corollary}
For $n>0,$ we have%
\begin{equation}
\sum_{k=0}^{n}\binom{n}{k}\left(  -1\right)  ^{k}F_{k}F_{n-k}=\left\{
\begin{array}
[c]{cc}%
0 & ;n\text{ is odd}\\
\frac{4}{3}F_{n} & ;n\text{ is even}%
\end{array}
.\right.  \label{17}%
\end{equation}

\end{corollary}

Finally, we obtain a new explicit formula for Fubini polynomials and numbers
in the following theorem.

\begin{theorem}
\label{teo5}For $y\neq\frac{-1}{2},$
\begin{equation}
F_{n}\left(  y\right)  =\sum_{k=0}^{n}%
\genfrac{\{}{\}}{0pt}{}{n}{k}%
k!y^{k}\frac{\left[  2^{n+1}\left(  y+1\right)  y^{k}+\left(  -1\right)
^{k+1}\right]  }{\left(  2y+1\right)  ^{k+1}}. \label{84}%
\end{equation}
When $y=1$ this becomes%
\begin{equation}
F_{n}=\sum_{k=0}^{n}%
\genfrac{\{}{\}}{0pt}{}{n}{k}%
k!\frac{\left[  2^{n+2}+\left(  -1\right)  ^{k+1}\right]  }{3^{k+1}},\text{
\ }n\geq0. \label{85}%
\end{equation}
When $y=-2$ this becomes%
\begin{equation}
F_{n}=\sum_{k=0}^{n}\left(  -1\right)  ^{n-k}%
\genfrac{\{}{\}}{0pt}{}{n}{k}%
k!\frac{2^{k-1}\left[  2^{n+k+1}+1\right]  }{3^{k+1}}\text{ \ }n>0. \label{86}%
\end{equation}

\end{theorem}

\begin{proof}
If we take $\frac{1}{y^{2}-1}$ in place of $y$ in (\ref{10}) we arrive at%
\begin{equation}
\frac{1}{1-\frac{1}{y^{2}-1}\left(  e^{2t}-1\right)  }=\frac{y^{2}-1}{2y^{2}%
}\left[  \frac{1}{y-e^{t}}+\frac{1}{y+e^{t}}\right]  . \label{79}%
\end{equation}
Each of the function in the above equation can be written as
\begin{align}
\frac{1}{1-\frac{1}{y^{2}-1}\left(  e^{2t}-1\right)  }  &  =\sum_{n=0}%
^{\infty}2^{n}F_{n}\left(  \frac{1}{y^{2}-1}\right)  \frac{t^{n}}%
{n!},\label{80}\\
\frac{1}{y-e^{t}}  &  =\frac{y}{y-1}\sum_{n=0}^{\infty}F_{n}\left(  \frac
{1}{y-1}\right)  \frac{t^{n}}{n!},\label{81}\\
\frac{1}{y+e^{t}}  &  =\frac{y}{y+1}\sum_{n=0}^{\infty}F_{n}\left(  \frac
{-1}{y+1}\right)  \frac{t^{n}}{n!}. \label{82}%
\end{align}
By equating the coefficients of $\frac{t^{n}}{n!},$ we have%
\begin{equation}
F_{n}\left(  y\right)  =2^{n+1}\left(  1+y\right)  F_{n}\left(  \frac{y^{2}%
}{1+2y}\right)  -\left(  1+2y\right)  F_{n}\left(  -y\right)  . \label{24}%
\end{equation}
Finally, using (\ref{2}) in the right hand side of the above equation yields
(\ref{84}).
\end{proof}

\section{Integrals of products of Fubini polynomials}

The integrals of products of various polynomials and functions have been
studied (e.g., \cite{Agoh2, B2, Dagli1, Kim, Liu}). In this section we deal
with an integral for a product of two Fubini polynomials. First we need the
following Lemma \ref{lem1} and Lemma \ref{lem2}.

\begin{lemma}
\label{lem1}For all $k\geq0$ and $n\geq1$ we have%
\begin{equation}%
{\displaystyle\int\limits_{-1}^{0}}
y^{k}F_{n}\left(  y\right)  dy=\frac{\left(  -1\right)  ^{k}}{k!}\sum
_{j=0}^{k}%
\genfrac{[}{]}{0pt}{}{k+1}{j+1}%
B_{n+j}, \label{25}%
\end{equation}
where $%
\genfrac{[}{]}{0pt}{}{n}{k}%
$ is the Stirling numbers of the fist kind (\cite{Graham}).
\end{lemma}

\begin{proof}
We prove (\ref{25}) by induction on $k.$ The case $k=0$ of (\ref{25}) is known
from (\ref{26}). If we integrate both sides of (\ref{11}) with respect to $y$
from $-1$ to $0$ and apply integration by parts, we have
\begin{align*}%
{\displaystyle\int\limits_{-1}^{0}}
F_{n+1}\left(  y\right)  dy  &  =%
{\displaystyle\int\limits_{-1}^{0}}
y\frac{d}{dy}\left[  F_{n}\left(  y\right)  +yF_{n}\left(  y\right)  \right]
\\
&  =\left[  y\left(  F_{n}\left(  y\right)  +yF_{n}\left(  y\right)  \right)
\right]  _{-1}^{0}-%
{\displaystyle\int\limits_{-1}^{0}}
\left[  F_{n}\left(  y\right)  +yF_{n}\left(  y\right)  \right]  dy.
\end{align*}
So using (\ref{26}) yields the case $k=1$ of (\ref{25}) as
\[%
{\displaystyle\int\limits_{-1}^{0}}
yF_{n}\left(  y\right)  dy=-\left(  B_{n+1}+B_{n}\right)  .
\]
Multiplying both sides of (\ref{11}) with $y$ and integrating it with respect
to $y$ from $-1$ to $0$ we obtain
\[%
{\displaystyle\int\limits_{-1}^{0}}
yF_{n+1}\left(  y\right)  dy=%
{\displaystyle\int\limits_{-1}^{0}}
y^{2}\frac{d}{dy}\left[  F_{n}\left(  y\right)  +yF_{n}\left(  y\right)
\right]  .
\]
Applying integration by parts and using (\ref{26}) yields the case $k=2$ of
(\ref{25}) as%
\[
2%
{\displaystyle\int\limits_{-1}^{0}}
y^{2}F_{n}\left(  y\right)  dy=B_{n+2}+3B_{n+1}+2B_{n}.
\]
If we multiply both sides of (\ref{11}) with $y^{k}$ and integrating it with
respect to $y$ from $-1$ to $0$ we obtain
\[%
{\displaystyle\int\limits_{-1}^{0}}
y^{k}F_{n+1}\left(  y\right)  dy=%
{\displaystyle\int\limits_{-1}^{0}}
y^{k+1}\frac{d}{dy}\left[  F_{n}\left(  y\right)  +yF_{n}\left(  y\right)
\right]  .
\]
Applying integration by parts to the right hand side of the above equation and
considering%
\[%
{\displaystyle\int\limits_{-1}^{0}}
y^{k}F_{n}\left(  y\right)  dy=\frac{\left(  -1\right)  ^{k}}{k!}\sum
_{j=0}^{k}%
\genfrac{[}{]}{0pt}{}{k+1}{j+1}%
B_{n+j},
\]
we have%
\begin{align*}
&
{\displaystyle\int\limits_{-1}^{0}}
y^{k+1}F_{n+1}\left(  y\right)  dy\\
&  \quad=\frac{\left(  -1\right)  ^{k+1}}{\left(  k+1\right)  !}\sum_{j=0}^{k}%
\genfrac{[}{]}{0pt}{}{k+1}{j+1}%
B_{n+j+1}+\frac{\left(  -1\right)  ^{k+1}}{\left(  k+1\right)  !}\sum
_{j=0}^{k}\left(  k+1\right)
\genfrac{[}{]}{0pt}{}{k+1}{j+1}%
B_{n+j}.
\end{align*}
Finally, the well known relations
\[%
\genfrac{[}{]}{0pt}{}{n+1}{k}%
=n%
\genfrac{[}{]}{0pt}{}{n}{k}%
+%
\genfrac{[}{]}{0pt}{}{n}{k-1}%
\text{ and }%
\genfrac{[}{]}{0pt}{}{n}{1}%
=\left(  n-1\right)  !,
\]
give that the statement is true for $k+1.$
\end{proof}

\begin{lemma}
\label{lem2}For any non-negative integer $m$ and $j$,
\begin{equation}
\sum_{k=j}^{m}%
\genfrac{\{}{\}}{0pt}{}{m}{k}%
\genfrac{[}{]}{0pt}{}{k+1}{j+1}%
(-1)^{k}=(-1)^{m}\binom{m}{j}. \label{33}%
\end{equation}

\end{lemma}

\begin{proof}
We rewrite this equation into matrix form by using the matrices
\[
(\mathcal{S}_{1})_{i,j}=(-1)^{i+j}%
\genfrac{[}{]}{0pt}{}{i+1}{j+1}%
,\quad(\mathcal{S}_{2})_{i,j}=%
\genfrac{\{}{\}}{0pt}{}{i}{j}%
,\quad(\mathcal{B})_{i,j}=\binom{i}{j}.
\]
These can be considered as infinite matrices so that the statement we are
going to prove takes the form
\[
\mathcal{S}_{2}\mathcal{S}_{1}=\mathcal{B}^{-1},
\]
as the elementwise inverse of the matrix $\mathcal{B}$ is $(\mathcal{B}%
)_{i,k}^{-1}=(-1)^{i+k}\binom{i}{k}$. The above equation is equivalent to
\[
\mathcal{S}_{1}=\mathcal{B}^{-1}\mathcal{S}_{2}^{-1}=\left(  \mathcal{S}%
_{2}\mathcal{B}\right)  ^{-1}.
\]
The matrix on the right hand side is easily decipherable. Elementwise it is
\[
(\left(  \mathcal{S}_{2}\mathcal{B}\right)  ^{-1})_{i,j}=\sum_{k=0}^{i}%
\genfrac{\{}{\}}{0pt}{}{i}{k}%
\binom{k}{j}.
\]
The latter sum simply equals to
\[
\sum_{k=0}^{i}%
\genfrac{\{}{\}}{0pt}{}{i}{k}%
\binom{k}{j}=%
\genfrac{\{}{\}}{0pt}{}{i+1}{j+1}%
,
\]
as it is known (see \cite[p. 251, formula (6.15)]{Graham}). Hence our original
statement equals to the matrix equation
\[
(\mathcal{S}_{1})_{i,j}^{-1}=%
\genfrac{\{}{\}}{0pt}{}{i+1}{j+1}%
.
\]
This is nothing else but the reformulation of the fact that the second and
signed first kind Stirling matrices are inverses of each other.
\end{proof}

Now, we are ready to give the integrals of products of Fubini polynomials.
Using (\ref{2}) we have%
\[%
{\displaystyle\int\limits_{-1}^{0}}
F_{m}\left(  y\right)  F_{n}\left(  y\right)  dy=%
{\displaystyle\int\limits_{-1}^{0}}
\sum_{k=0}^{m}%
\genfrac{\{}{\}}{0pt}{}{m}{k}%
k!y^{k}F_{n}\left(  y\right)  dy.
\]
Then, interchanging the sum and integral in the above equation and using
(\ref{25}) yield
\[%
{\displaystyle\int\limits_{-1}^{0}}
F_{m}\left(  y\right)  F_{n}\left(  y\right)  dy=\sum_{j=0}^{m}\sum_{k=j}^{m}%
\genfrac{\{}{\}}{0pt}{}{m}{k}%
\genfrac{[}{]}{0pt}{}{k+1}{j+1}%
\left(  -1\right)  ^{k}B_{n+j}.
\]
Finally, using Lemma \ref{lem2} gives the following theorem.

\begin{theorem}
\label{teo2}For all $m\geq0$ and $n\geq1$ we have%
\begin{equation}%
{\displaystyle\int\limits_{-1}^{0}}
F_{m}\left(  y\right)  F_{n}\left(  y\right)  dy=(-1)^{m}\sum_{j=0}^{m}%
\binom{m}{j}B_{n+j}. \label{30}%
\end{equation}

\end{theorem}

Using the representation (\ref{2}) in (\ref{30}) and integrating termwise one
obtains
\[
\sum_{k=0}^{n}\sum_{j=0}^{m}%
\genfrac{\{}{\}}{0pt}{}{n}{k}%
\genfrac{\{}{\}}{0pt}{}{m}{j}%
\frac{\left(  -1\right)  ^{k+j}k!j!}{k+j+1}=(-1)^{m}\sum_{j=0}^{m}\binom{m}%
{j}B_{n+j}.
\]
This double sum identity extends (\ref{32}).

In order to give an application of Lemma \ref{lem1}, now we emphasize the
summation in the right hand of (\ref{25}). Rahmani \cite{Rahmani} defined
$p$-Bernoulli numbers as
\[
\sum_{n=0}^{\infty}B_{n,p}\frac{t^{n}}{n!}=\text{ }_{2}F_{1}\left(
1,1;p+2;1-e^{t}\right)  ,
\]
where $_{2}F_{1}\left(  a,b;c;z\right)  $ denotes the Gaussian hypergeometric
function \cite{Andrews}. These numbers can be written in terms Stirling
numbers of the first kind
\[
\sum_{j=0}^{p}\left(  -1\right)  ^{j}%
\genfrac{[}{]}{0pt}{}{p}{j}%
B_{n+j}=\frac{p!}{p+1}B_{n,p},\text{ \ }n,p\geq0.
\]
From the above equation, we have
\begin{equation}
\sum_{j=0}^{p}\left(  -1\right)  ^{j+1}%
\genfrac{[}{]}{0pt}{}{p+1}{j+1}%
B_{n+j}=\frac{\left(  p+1\right)  !}{p+2}B_{n-1,p+1},\text{ \ }n\geq1,p\geq0.
\label{27}%
\end{equation}
Moreover when $n$\ is odd or even we have
\[
\left(  -1\right)  ^{j+1}B_{n+j}=B_{n+j}\ \text{or }\left(  -1\right)
^{j+1}B_{n+j}=-B_{n+j},\text{ \ }n>1,
\]
\ respectively. Therefore we have%
\[
\sum_{j=0}^{p}%
\genfrac{[}{]}{0pt}{}{p+1}{j+1}%
B_{n+j}=\left\{
\begin{array}
[c]{cc}%
\frac{\left(  p+1\right)  !}{p+2}B_{n-1,p+1}, & n\text{ is odd}\\
-\frac{\left(  p+1\right)  !}{p+2}B_{n-1,p+1}, & n\text{ is even}%
\end{array}
\right.  .
\]
Using the above equation, (\ref{25}) can be written as
\begin{equation}%
{\displaystyle\int\limits_{-1}^{0}}
y^{p}F_{n}\left(  y\right)  dy=\left\{
\begin{array}
[c]{cc}%
\left(  -1\right)  ^{p}\frac{p+1}{p+2}B_{n-1,p+1}, & n\text{ is odd}\\
\left(  -1\right)  ^{p+1}\frac{p+1}{p+2}B_{n-1,p+1}, & n\text{ is even}%
\end{array}
\right.  , \label{28}%
\end{equation}
where $n>1,p\geq0.$ On the other hand, using (\ref{2}) in the left part of
(\ref{28}), a new explicit formula for $p$-Bernoulli numbers is obtained.

\begin{theorem}
For $n>1$ and $p>0$,%
\[
B_{2n-1,p}=\frac{p+1}{p}\sum_{k=0}^{2n-1}%
\genfrac{\{}{\}}{0pt}{}{2n-1}{k+1}%
\frac{\left(  -1\right)  ^{k+1}\left(  k+1\right)  !}{k+p+1}%
\]
and%
\[
B_{2n,p}=\frac{p+1}{p}\sum_{k=0}^{2n}%
\genfrac{\{}{\}}{0pt}{}{2n+1}{k+1}%
\frac{\left(  -1\right)  ^{k}\left(  k+1\right)  !}{k+p+1}.
\]

\end{theorem}

\section{Applications}

Apostol-Bernoulli functions $\mathcal{B}_{n}\left(  \lambda\right)  $ have the
following explicit expression
\begin{equation}
\mathcal{B}_{n}\left(  \lambda\right)  =\frac{n}{\lambda-1}\sum_{k=0}^{n-1}%
\genfrac{\{}{\}}{0pt}{}{n-1}{k}%
k!\left(  \frac{\lambda}{1-\lambda}\right)  ^{k},\,\ \ \lambda\in%
\mathbb{C}
\backslash\{1\}. \label{53}%
\end{equation}
Thus for $\lambda\neq1,$
\[
\mathcal{B}_{0}\left(  \lambda\right)  =0,\text{ }\mathcal{B}_{1}\left(
\lambda\right)  =\frac{1}{\lambda-1},\text{ }\mathcal{B}_{2}\left(
\lambda\right)  =\frac{-2\lambda}{\left(  \lambda-1\right)  ^{2}%
},...\text{etc.}%
\]
The functions $\mathcal{B}_{n}\left(  \lambda\right)  $ are rational functions
in the second variable, $\lambda$. These functions were introduced by Apostol
\cite{APOSTOL} in order to evaluate the Lerch transcendent (also Lerch zeta
function) for negative integer values of $s$ and also were studied and
generalized recently in a number of papers, under the name Apostol-Bernoulli numbers.

Comparing the (\ref{53}) to (\ref{2})$,$ Apostol-Bernoulli functions can be
expressed by Fubini polynomials as (\cite{B3})%
\begin{equation}
\mathcal{B}_{n+1}\left(  \lambda\right)  =\frac{n+1}{\lambda-1}F_{n}\left(
\frac{\lambda}{1-\lambda}\right)  ,\text{ \ \ \ }\lambda\in%
\mathbb{C}
\backslash\{1\}. \label{22}%
\end{equation}

We can use this relation to obtain some new properties of $\mathcal{B}%
_{n}\left(  \lambda\right)  .$ For example setting $y=\frac{-\lambda}%
{\lambda-1}$ in (\ref{21}) we have
\[
\frac{\mathcal{B}_{n+1}\left(  \lambda\right)  }{\left(  n+1\right)  }=\left(
-1\right)  ^{n}\lambda\sum_{k=0}^{n}%
\genfrac{\{}{\}}{0pt}{}{n}{k}%
k!\left(  \frac{1}{\lambda-1}\right)  ^{k+1},\text{ \ }\lambda\neq1,\text{
}n\geq0,
\]
which was obtained in \cite[Thereom 4.3]{GuoandQi}. Similarly, from Thereom
\ref{teo1} we get the sums of products of Apostol-Bernoulli functions as given
in \cite[Corollary 1.3]{KimandHu} by different method. Moreover, using the
equation (\ref{84}) of Theorem \ref{teo5} gives a new explicit formula for
Apostol-Bernoulli functions.

\begin{corollary}
For $\lambda\neq\pm1$ and $n\geq0,$%
\[
\frac{\mathcal{B}_{n+1}\left(  \lambda\right)  }{\left(  n+1\right)  }%
=\sum_{k=0}^{n}%
\genfrac{\{}{\}}{0pt}{}{n}{k}%
k!\frac{\left(  -\lambda\right)  ^{k}\left[  2^{n+1}\lambda^{k}+\left(
\lambda-1\right)  ^{k+1}\right]  }{\left(  \lambda^{2}-1\right)  ^{k+1}}.
\]

\end{corollary}

To give a different application of the relation (\ref{22}), first we deal with
Lemma \ref{lem1}. Replacing $y$ with $\frac{\lambda}{1-\lambda}$ in
(\ref{25}), we have
\[%
{\displaystyle\int\limits_{-\infty}^{0}}
\frac{\lambda^{k}}{\left(  \lambda-1\right)  ^{k+1}}\mathcal{B}_{n+1}\left(
\lambda\right)  d\lambda=\frac{n+1}{k!}\sum_{j=0}^{k}%
\genfrac{[}{]}{0pt}{}{k+1}{j+1}%
B_{n+j},
\]
where $k\geq0$ and $n\geq1.$ Similarly, from Theorem \ref{teo2} we obtain the
integrals of products of Apostol-Bernoulli functions as given in the following corollary.

\begin{corollary}
For all $m\geq0$ and $n\geq1$ we have
\[%
{\displaystyle\int\limits_{-\infty}^{0}}
\mathcal{B}_{m}\left(  \lambda\right)  \mathcal{B}_{n}\left(  \lambda\right)
d\lambda=(-1)^{m}\left(  m+1\right)  \left(  n+1\right)  \sum_{j=0}^{m}%
\binom{m}{j}B_{n+j}.
\]

\end{corollary}

\end{document}